\documentclass{article}
\usepackage[utf8]{inputenc}
\usepackage{amsmath}
\usepackage{amsthm}
\usepackage{amssymb}
\usepackage[backref=page,colorlinks=true]{hyperref}

\newtheorem{theorem}{Theorem}
\newtheorem{lemma}[theorem]{Lemma}

\newtheorem{example}[theorem]{Example}

\newtheorem{definition}[theorem]{Definition}

\usepackage{enumitem}

\newcommand{\diam}{\ensuremath{\,\mathrm{diam}}}

\newcommand{\la}{\langle}
\newcommand{\ra}{\rangle}

\title{Lower Bounds on the Low-Distortion Embedding Dimension of Submanifolds of $\mathbb{R}^n$}

\author{Mark Iwen\thanks{Michigan State University, Department of Mathematics, and the Department of Computational Mathematics, Science and Engineering (CMSE), \texttt{markiwen@math.msu.edu}.  Supported in part by NSF DMS 1912706.}, 
Benjamin Schmidt\thanks{Michigan State Univeristy, Department of Mathematics, \texttt{schmidt@math.msu.edu}.  Supported in part by a Simons Collaboration Grant.}, 
Arman Tavakoli\thanks{Michigan State University, Department of Mathematics, \texttt{tavakol4@msu.edu}.  Supported in part by NSF DMS 1912706, and by an MSU College of Natural Science Dissertation Completion Fellowship.}
}

\begin{document}

\maketitle

\begin{abstract}
Let $\mathcal{M}$ be a smooth submanifold of   $\mathbb{R}^n$ equipped with the Euclidean (chordal) metric.  This note considers the smallest dimension $m$ for which there exists a bi-Lipschitz function $f: \mathcal{M} \mapsto \mathbb{R}^m$ with bi-Lipschitz constants close to one.  The main result bounds the embedding dimension $m$ below in terms of the bi-Lipschitz constants of $f$ and the reach, volume, diameter, and dimension of $\mathcal{M}$.  This new lower bound is applied to show that prior upper bounds by Eftekhari and Wakin \cite{eftekhari_new_2015} on the minimal low-distortion embedding dimension of such manifolds using random matrices achieve near-optimal dependence on both reach and volume.  This supports random linear maps as being nearly as efficient as the best possible nonlinear maps at reducing the ambient dimension for manifold data.  In the process of proving our main result, we also prove similar results concerning the impossibility of achieving better nonlinear measurement maps with the Restricted Isometry Property (RIP) in compressive sensing applications.
\end{abstract}

\section{Introduction}

Given a set $T \subset \mathbb{R}^n$ and $\epsilon \in (0,1)$ we consider functions $f: T \mapsto \mathbb{R}^m$ satisfying 
$$(1 - \epsilon) \| {\bf x}- {\bf y} \|_2 \leq \| f( {\bf x}) - f({\bf y}) \|_2 \leq (1 + \epsilon) \| {\bf x}- {\bf y} \|_2$$
for all ${\bf x}, {\bf y} \in T$, where $\| \cdot \|_2$ denotes the $\ell_2$-norm.  Such a function $f$ is said to be an $\epsilon$-distortion (Euclidean) embedding of $T$ into $\mathbb{R}^m$ with embedding dimension $m$.  The use of such embeddings has exploded in growth during the last two decades in computer science (where $T$ is often a finite point set), numerical linear algebra (where $T$ is often a low-dimensional subspace), and signal processing (where $T$ is often a manifold, or all $s$-sparse vectors in $\mathbb{R}^n$).  This is due in large part to results of Johnson and Lindenstrauss (JL) \cite{lindenstrauss1984extensions} and their later simplifications (see, e.g., \cite{dasgupta1999elementary,achlioptas2003database}), as well as subsequent developments in compressive sensing \cite{foucart_mathematical_2013} which showed that random matrices could easily provide such embeddings.  Questions concerning how sub-optimal such linear embeddings might be at reducing the low-distortion embedding dimension $m$ in comparison to the best possible nonlinear $f$ have remained less well explored, however see, e.g., \cite{larsen2017optimality} for a rare result of this kind concerning finite sets $T$.\\

In this note we are principally interested in the optimality of random matrices as $\epsilon$-distortion embeddings of submanifolds of $\mathbb{R}^n$ into $\mathbb{R}^m$.  More specifically, we will consider the optimality of the following result with respect to bounding the low-distortion embedding dimension $m$.

\begin{theorem}[Theorem 2 in \cite{eftekhari_new_2015}]
\label{Thm:Wakin}
Let $\mathcal{M}$ be a compact $d$-dimensional Riemannian submanifold of $\mathbb{R}^n$ having reach $\tau$ and volume $V$.  Conveniently assume that 
\begin{equation}
\label{equ:WakinAssumption}
\frac{V}{\tau^d} \geq \left( \frac{21}{2 \sqrt{d}} \right)^d.
\end{equation}
Fix $0 < \epsilon \leq 1/3$ and $0 < \rho < 1$.  Let $\Phi$ be a random $m \times n$ matrix populated with i.i.d. zero-mean Gaussian random variables with variance of $1/m$ with 
$$m \geq 18 \epsilon^{-2} \max \left( 24 d + 2d \log \left( \frac{\sqrt{d}}{\tau \epsilon^2} \right) + \log(2 V^2),~\log \left(\frac{8}{\rho} \right) \right).$$
Then with probability at least $1-\rho$ the following statement holds:  For every pair of points ${\bf p}, {\bf q} \in \mathcal{M}$,
$$(1 - \epsilon) \| {\bf p}- {\bf q} \|_2 \leq \| \Phi {\bf p}- \Phi {\bf q} \|_2 \leq (1 + \epsilon) \| {\bf p}- {\bf q} \|_2.$$
\end{theorem}

The main result of this paper is a lower bound which demonstrates the near-optimality of  Theorem~\ref{Thm:Wakin}.  In particular, we will see that it is generally necessary for $m$ to exhibit a logarithmic dependence on both the volume $V$ and the inverse of reach $1/\tau$ when \eqref{equ:WakinAssumption} holds.

\subsection{The Main Result and Some Examples}

In this paper we prove the following lower bound for the low-distortion embedding dimension of a smooth submanifold of $\mathbb{R}^n$.

\begin{theorem}[Main Result]
\label{MAINTHM}
Let $\mathcal{M}$ be a $d$-dimensional smooth submanifold of $\mathbb{R}^n$, possibly with boundary, that has volume $V$ and reach $\tau$.\footnote{See Definition~\ref{def:Reach} for the formal definition of reach.} Furthermore, let $\epsilon \in (0,1)$, and suppose that $f: \mathcal{M} \mapsto \mathbb{R}^m$ satisfies 
$$(1 - \epsilon) \| {\bf p} - {\bf q} \|_2 \leq \| f({\bf p}) - f({\bf q}) \|_2 \leq (1 + \epsilon) \| {\bf p} - {\bf q} \|_2$$
for all ${\bf p}, {\bf q} \in \mathcal{M}$.  If $\frac{1}{4\sqrt{e}}(\frac{V}{\omega_d}) ^{\frac{1}{d}} \leq \tau$ then

\begin{align*}
m \geq C_1 \left( \frac{1-\epsilon}{1+\epsilon} \right)^2 \frac{d}{\mathrm{diam}^2(\mathcal{M})} \left(\frac{V}{\omega_d} \right)^{\frac{2}{d}}.
\end{align*}

Else, if $\tau \leq \frac{1}{4\sqrt{e}}(\frac{V}{\omega_d}) ^{\frac{1}{d}}$ then

\begin{align*}
m \geq C_2  \left(\frac{1-\epsilon}{1+\epsilon} \right)^2 \frac{d \tau^2 }{\mathrm{diam}^2(\mathcal{M})} \log \left(\frac{V^{\frac{1}{d}}}{(\omega_d)^{\frac{1}{d}} (4\tau)} \right).
\end{align*}

Here $C_1, C_2 > 0$ are universal constants that are independent of all other quantities, $\omega_{d} := \frac{\pi^{d/2}}{\Gamma(\frac{d}{2}+1)}$ is the volume of the $d$-dimensional Euclidean unit ball, and $\mathrm{diam}(\mathcal{M}) := \sup_{{\bf p}, {\bf q} \in \mathcal{M}} \| {\bf p}- {\bf q} \|_2$.
\end{theorem}

To help illustrate the use of Theorem~\ref{MAINTHM} we will now apply it to the standard examples of the $d$-dimensional sphere of radius $r$, $r S^d$ and $\ell_2$-ball $r B^d_2$ as submanifolds of $\mathbb{R}^n$.  In the process we will see that it gives near-optimal (up to constants) lower bounds on $m$ for these simple examples when $\epsilon$ is small.

\begin{example}[The Low-Distortion Embedding Dimension of Euclidean ball of radius $r$, $r B^d_2$]
\label{example:ball}
The pertinent geometric parameters of $r B^d_2 \subset \mathbb{R}^n$ as a submanifold of $\mathbb{R}^n$ are as follows:  The volume of $r B^d_2$ is $r^d \omega_{d}$, its reach is $\tau(r B^d_2) = \infty$, and its diameter is $\mathrm{diam}(r B^d_2) = 2r$.  Applying  Theorem~\ref{MAINTHM} to $\mathcal{M} = r B^d_2$ (i.e., a ball of radius $r$ in an $d$-dimensional subspace of $\mathbb{R}^n$) we therefore learn that 
$$m \geq C_1 \left( \frac{1-\epsilon}{1+\epsilon} \right)^2 \frac{d}{(2r)^2} \left(\frac{r^d \omega_{d}}{\omega_d} \right)^{\frac{2}{d}} ~=~ \frac{C_1}{4} \left( \frac{1-\epsilon}{1+\epsilon} \right)^2 d.$$
This result agrees with our intuition based on the fact that for a no distortion $\epsilon = 0$ embedding of $r B^d$, a target embedding dimension of $d$ is both necessary and sufficient for all choices of the radius $r$. 
\end{example}

\begin{example}[The Low-Distortion Embedding Dimension of a Sphere $r S^d$]
\label{example:sphere}
The pertinent geometric parameters of $r S^d \subset \mathbb{R}^n$ as a submanifold of $\mathbb{R}^n$ are as follows:  The volume of $r S^d$ is $2 r^d \frac{\pi^{\frac{d+1}{2}}}{\Gamma(\frac{d+1}{2})}$, its reach is $\tau(r S^d) = r$, and its diameter is $\mathrm{diam}({r S^d}) = 2r$.  Check the relation between reach and the other parameters to find that 
$$\frac{1}{4\sqrt{e}} \left(\frac{V}{\omega_d}\right) ^{\frac{1}{d}} =
\frac{r}{4\sqrt{e}}\left(2\sqrt{\pi}\frac{\Gamma(\frac{d}{2}+1)}{\Gamma(\frac{d+1}{2})}\right) ^{\frac{1}{d}} $$
Using that $\displaystyle \lim_{d\rightarrow \infty} \left(2\sqrt{\pi}\frac{\Gamma(\frac{d}{2}+1)}{\Gamma(\frac{d+1}{2})}\right) ^{\frac{1}{d}} = 1$, for large enough $d$ we have $\frac{1}{4\sqrt{e}} \left(\frac{V}{\omega_d} \right) ^{\frac{1}{d}} < r = \tau$.  Applying Theorem~\ref{MAINTHM} we now learn that 
$$m \geq \frac{C_1}{4} \left( \frac{1-\epsilon}{1+\epsilon} \right)^2 d \left( 2 \sqrt{\pi} \frac{\Gamma(\frac{d}{2}+1)}{\Gamma(\frac{d+1}{2})} \right)^{\frac{2}{d}}$$
holds for all $r$ and sufficiently large $d$.  This again agrees which our intuition that to embed $r S^d$ without distortion requires at least $d$ dimensions. 
\end{example}

The attentive reader may have noticed that both Examples~\ref{example:ball} and~\ref{example:sphere} end up utilizing the first $\frac{1}{4\sqrt{e}}(\frac{V}{\omega_d}) ^{\frac{1}{d}} \leq \tau$ case of Theorem~\ref{MAINTHM}.  Somewhat crucially, the following discussion will largly hinge on the second case.

\subsection{A Discussion of Theorems~\ref{Thm:Wakin} and~\ref{MAINTHM}}

Considering the second case of Theorem~\ref{MAINTHM} with $\tau \leq \frac{1}{4\sqrt{e}}(\frac{V}{\omega_d}) ^{\frac{1}{d}}$ we note that the ratio $\frac{V}{\tau^d}$ must be at least $\omega_d (4 \sqrt{e})^d = \frac{(4 \sqrt{e \pi})^d}{\Gamma \left(1 + d/2 \right)}$, where $\Gamma$ denotes the gamma function. 
Recalling Stirling's approximation one can show that $\left( \frac{d}{2e} \right)^{d/2 } \leq \Gamma \left( 1+d/2 \right)$.  As a result, we see, e.g., that assuming $\tau \leq \frac{1}{4\sqrt{e}}((\frac{d}{2e\pi})^{d/2}V)^{1/d} \leq \frac{1}{4\sqrt{e}} \left( \frac{\Gamma \left( 1+d/2 \right)}{\pi^{d/2}} V \right)^{1/d} = \frac{1}{4\sqrt{e}}(\frac{V}{\omega_d}) ^{\frac{1}{d}}$, 
or equivalently that $\left( \frac{32 e^2 \pi}{d} \right)^{d/2} =  \frac{(4\sqrt{e})^d}{(\frac{d}{2e\pi})^{d/2}} \leq \frac{V}{\tau^d}$,
implies the inequality $\tau \leq \frac{1}{4\sqrt{e}}(\frac{V}{\omega_d}) ^{\frac{1}{d}}$ in Theorem~\ref{MAINTHM}. 
We further note that this range of $\tau$ is also compatible with assumption \eqref{equ:WakinAssumption} in Theorem~\ref{Thm:Wakin} that $\frac{V}{\tau^d} \geq (\frac{21}{2 \sqrt{d}})^d$.\\  

Assuming then for the sake of comparison that $\frac{V}{\tau^d} \geq \max \left\{ (\frac{21}{2 \sqrt{d}})^d, \omega_d (4 \sqrt{e})^d \right\}$ and setting $\epsilon = \rho = 1/3$ (for example) in Theorem~\ref{Thm:Wakin} one can see that the minimum achievable, e.g., $\epsilon = 1/3$ embedding dimension $m$ of $\mathcal{M}$ must satisfy
$$\frac{C'_2 d \tau^2}{\mathrm{diam}^2(M)} \log \left(\frac{V^{\frac{1}{d}}}{(\omega_d)^{\frac{1}{d}} (4\tau)}\right) \underset{Thm~\ref{MAINTHM}}{\leq} m  \underset{Thm~\ref{Thm:Wakin}}{\leq} \tilde{C} d \left(1 + \log \left(\frac{d}{\tau^2} V^{\frac{2}{d}}\right) \right)$$
for all $d \geq 1$.  Simplifying the expression further under the assumption that, e.g.,  $\frac{V}{\tau^d} = (\frac{C'}{\sqrt{d}})^d$ we have that 
$$\frac{C'_2 \tau^2}{\mathrm{diam}^2(M)} \log \left(\frac{C'}{4 \sqrt{d} \sqrt[d]{\omega_d}}\right) \leq \frac{m}{d}  \leq \tilde{C} \left(1 + \log \left(2 C'^2\right) \right)$$
In this case we can see that under the assumptions above that the upper and lower bounds on $m$ differ by a factor proportional to $\frac{ \tau^2}{\mathrm{diam}^2(M)}$ for $d$ not too large (e.g., $d < 100$).\\

Taken all together we can see that Theorem~\ref{MAINTHM} supports  Theorem~\ref{Thm:Wakin} as being optimal up to (at worst) a factor proportional to $\frac{ \tau^2}{\mathrm{diam}^2(M)}$ for $d$ not too large, and $\frac{V}{\tau^d} = (\frac{C'}{\sqrt{d}})^d$ for appropriate ranges of $C' \in \mathbb{R}^+$.  This is perhaps most interesting due to the fact that Theorem~~\ref{MAINTHM} implies that it is not possible to improve Theorem~\ref{Thm:Wakin} beyond this factor in this regime, even by utilizing {\it nonlinear} functions $f$ in the place of the matrix $\Phi$.  Indeed, a similar fact is established as a consequence of the proof of Theorem~\ref{MAINTHM} for matrices with the Restricted Isometry Property (RIP) (see Example~\ref{RIPexample} below), and an existing result of this kind is also reproduced for JL embeddings of finite sets (see Example~\ref{example:JLlowerbound}).

\subsection{An Outline of the Proof of Theorem~\ref{MAINTHM}}

The remainder of the paper is organized as follows.  In Section~\ref{sec:GenProbBound} we prove Theorem~\ref{FittingItInTheBall}, a general method for bounding the low-distortion embedding dimension of an arbitrary subset $T \subset \mathbb{R}^n$ in terms of its covering numbers.  Given this result the proof of Theorem~\ref{MAINTHM} is then reduced to lower bounding the covering numbers of submanifolds of $\mathbb{R}^n$ as subsets of $\mathbb{R}^n$.  This is done in Section~\ref{sec:CoverNumbound} and results in Theorem~\ref{thm:ManifoldCoverLowerBound}.  Finally, Theorem~\ref{MAINTHM} is then proven in Section~\ref{sec:ProofofMainThm}.

\section{General Lower Bounds for Embedding Dimension via High Dimensional Probability}
\label{sec:GenProbBound}

In this section we present a relatively simple lower bound for the achievable low-distortion Euclidean embedding dimension, $m$, of an arbitrary bounded subset $T \subset \mathbb{R}^n$ into $\mathbb{R}^m$ by any function.  As we shall see, when coupled with lower bounds for the covering numbers of $T$, Theorem~\ref{FittingItInTheBall} below allows one to quickly reproduce more specialized results concerning the necessary low-distortion embedding dimension of finite point sets \cite{alon2003problems,larsen_et_al:LIPIcs:2016:6203,larsen2017optimality} in certain parameter regimes.  Similarly, the result can also be used to quickly lower bound the number of rows  that any matrix with the RIP must have, reproducing existing lower bounds of interest in compressive sensing (see, e.g., \cite[Corollary 10.8]{foucart_mathematical_2013}).  In this second example one can further learn that existing RIP matrices are optimal in the sense that they match the necessary embedding dimension of even nonlinear embeddings of all $s$-sparse vectors in $\mathbb{R}^n$ (up to constants in certain parameter regimes).\\

To prove this general lower bound result we need the notion of the {\it Gaussian Width} of a set $T \subset \mathbb{R}^n$.

\begin{definition}\cite[Definition 7.5.1]{vershynin_high-dimensional_2018}
The Gaussian Width of a set $T \subset \mathbb{R}^n$ is  \begin{align*}
w(T) := \mathbb{E} \sup_{{\bf x} \in T} \, \langle {\bf g},{\bf x} \rangle
\end{align*}
where ${\bf g}$ is a random vector with $n$ independent and identically distributed (i.i.d.) mean $0$ and variance $1$ Gaussian entries.  
\end{definition}

The following lemma bounds the Gaussian width of the image of a low-distortion linear map of $T \subset \mathbb{R}^n$ in terms of the Gaussian width of $T$.  It is largely a consequence of the Sudakov-Fernique comparison theorem for mean 0 Gaussian processes (see, e.g., \cite[Theorem 7.2.11]{vershynin_high-dimensional_2018} for the statement used here).

\begin{lemma} \label{gaussianComparison}
Let $a,b \in \mathbb{R}^+$, $T \subset \mathbb{R}^n$, and $f:T \rightarrow \mathbb{R}^m$ be a function such that $a \|{\bf x}- {\bf y}\|_2 \leq \|f({\bf x})-f({\bf y})\|_2 \leq b \|{\bf x}- {\bf y}\|_2$
holds for all ${\bf x}, {\bf y} \in T$. Then, $a \, w(T) \leq w \left(f(T) \right) \leq b \, w(T)$ also holds.
\end{lemma}
\begin{proof}  Define the mean 0 Gaussian processes via $X_{\bf x} := a \la {\bf g}, {\bf x} \ra$, $Y_{\bf x} := \la {\bf g'}, f({\bf x}) \ra$, and $Z_{\bf x} = b \, \la {\bf g}, {\bf x} \ra$ for all ${\bf x} \in T$, where ${\bf g}, {\bf g'}$ are Gaussian random vectors with $n$ and $m$ i.i.d. mean $0$ and variance $1$ Gaussian entries, respectively.  For each ${\bf x}, {\bf y} \in T$, we then have that their increments satisfy

\begin{align*}
&\mathbb{E}\left(X_{\bf x}-X_{\bf y} \right)^2 = a^2 \mathbb{E} \la {\bf g}, {\bf x} - {\bf y} \ra^2 = a^2 \|{\bf x} - {\bf y}\|^2_2 \leq \\
&\mathbb{E}(Y_{\bf x}-Y_{\bf y})^2 = \mathbb{E} \la {\bf g'}, f({\bf x}) - f({\bf y}) \ra^2 = \|f({\bf x}) - f({\bf y})\|^2_2 \leq  \\
&\mathbb{E}(Z_{\bf x}-Z_{\bf y})^2 = b^2 \mathbb{E} \la {\bf g}, {\bf x} - {\bf y} \ra^2 = b^2 \|{\bf x} - {\bf y}\|^2_2. 
\end{align*}
The stated result now follows from the Sudakov-Fernique comparison theorem together with the fact that, e.g.,  $\mathbb{E}\sup_{{\bf x} \in T} X_{\bf x} = a \, \mathbb{E}\sup_{{\bf x} \in T} \la {\bf g}, {\bf x} \ra = a \,w(T)$.
\end{proof}

With Lemma~\ref{gaussianComparison} in hand we are now able to prove the main theorem of this section.  The proof also utilizes basic properties of Gaussian widths \cite[Proposition 7.5.2]{vershynin_high-dimensional_2018}, the Gaussian width of the open Euclidean unit ball in $m$ dimensions, $B^m_2 \subset \mathbb{R}^m$, being $w(B^m_2) \leq C' \sqrt{m}$ for an absolute constant $C' \in \mathbb{R}^+$ \cite[Example 7.5.7]{vershynin_high-dimensional_2018}, and Sudakov’s minoration inequality \cite[Theorem 7.4.1]{vershynin_high-dimensional_2018}.

\begin{theorem} \label{FittingItInTheBall} 
Let $a,b, \delta \in \mathbb{R}^+$, $T \subset \mathbb{R}^n$, and $f:T \rightarrow \mathbb{R}^m$ be a function such that $a \|{\bf x}- {\bf y}\|_2 \leq \|f({\bf x})-f({\bf y})\|_2 \leq b \|{\bf x}- {\bf y}\|_2$
holds for all ${\bf x}, {\bf y} \in T$. Then, 
\begin{align*}
     \frac{C a^2 \delta^2}{b^2} \, \frac{\log N(T,\delta)}{\mathrm{diam}^2(T)}
     \leq \left(\frac{2 a}{C' b} \frac{w(T)}{\mathrm{diam}(T)}\right)^2 \leq m
\end{align*}
holds, where $\mathrm{diam}(T) := \sup_{{\bf x}, {\bf y} \in T} \| {\bf x}- {\bf y} \|_2$, and where $N(T,\delta)$ is the $\delta$-covering number of $T$ by Euclidean balls of radius $\delta$ centered on $T$.  Here $C, C' > 0$ are universal constants that are independent of all other quantities.
\end{theorem}

\begin{proof}
Consider the set $f(T)-f(T) := \{ f({\bf x}) - f({\bf y}) ~|~ {\bf x}, {\bf y} \in T\}$.  By Lemma~\ref{gaussianComparison} and the properties of Gaussian widths, $a \, w(T) \, \leq w\left(f(T)\right) = \frac{1}{2} w\left(f(T) - f(T)\right)$.  Furthermore, $f(T) - f(T)$ is contained in a ball in $\mathbb{R}^m$ with radius $b \diam(T)$. Since $w(B^{m}_2) \leq C' \sqrt{m}$,  monotonicity of the Gaussian width and scaling properties now imply that 
$w(f(T)-f(T)) \leq C' \, b \, \mathrm{diam}(T) \sqrt{m}$. Therefore,
\begin{align*}
 a \, w(T) \leq \frac{1}{2}w\left(f(T) - f(T)\right) \leq \frac{C' \, b}{2} \, \diam(T) \, \sqrt{m}.
\end{align*}
The rightmost inequality follows.  The leftmost inequality is now a consequence of Sudakov’s minoration inequality.
\end{proof}

We will now use Theorem~\ref{FittingItInTheBall} to reproduce the promised lower bounds concerning RIP matrices and the low-distortion embedding dimension of finite point sets.  

\begin{example}[The Necessary Low-Distortion Embedding Dimension of Finite Sets]  
\label{example:JLlowerbound}
Let $\epsilon \in (0,1)$, and $T = \{ {\bf x}_1, \dots {\bf x}_N \} \subset \mathbb{R}^n$ be a finite set of points whose minimal distance between any two elements is $\delta := \min_{j \neq k} \| {\bf x}_j - {\bf x}_k \|_2$.  Note that the $\delta/2$-covering number of $T$ is $N(T, \delta/2) = N$.  As a result, Theorem~\ref{FittingItInTheBall} implies that any function $f:\mathbb{R}^n \rightarrow \mathbb{R}^m$ satisfying $\sqrt{1 - \epsilon} \|{\bf x}_j- {\bf x}_k\|_2 \leq \|f({\bf x}_j)-f({\bf x}_k)\|_2 \leq \sqrt{1 + \epsilon} \|{\bf x}_j- {\bf x}_k\|_2$
for all ${\bf x}_j, {\bf x}_k \in T$ must have 
$m \geq C'' \frac{(1-\epsilon) \delta^2}{(1+\epsilon)} \frac{\log N}{{\rm diam}^2(T)}$ for an absolute constant $C'' \in \mathbb{R}^+$. \\ 

Now consider $\epsilon$ not too small (e.g., $\epsilon > 0.01$). Furthermore, let $T$ to be a subset of, e.g., a $0.01$-packing of the closed unit ball $\overline{B_2^n} \subset \mathbb{R}^n$ so that $\delta > 0.01$ and $1.98 \leq {\rm diam}(T) \leq 2$ both hold.  For this range of $\epsilon$ and such $T$ coming from a packing of the unit ball one can see that the lower bound on $m$ from Theorem~\ref{FittingItInTheBall} matches (up to constants) the most recent worst case bounds of the form $m = \Omega \left( \log(N) / \epsilon^2 \right)$ proven by Larson and Nelson in \cite{larsen2017optimality} for finite point sets.  Furthermore, this example application of Theorem~\ref{FittingItInTheBall} helps to cement our intuition that (subsets of) packings of the unit ball are among the most difficult finite sets to embed into a lower dimensional Euclidean space.
\end{example}

\begin{example}[The Number of Rows Required by an RIP matrix, and the Non-Existence of Better Nonlinear Low-Distortion Embeddings for Sparse Vectors] 
\label{RIPexample}

Let $S \subseteq [n] := \{ 1, \dots, n \} \subset \mathbb{N}$, and denote the $j^{\rm th}$-column of the $n \times n$ identity matrix by ${\bf e}_j$ for all $j \in [n]$. Next, define $C_S := {\rm span} \{ {\bf e}_j ~|~ j \in S \}$ and let $\displaystyle U_s := \left( \bigcup_{S \subset [n],~|S| \leq s} C_S \right) \bigcap ~\overline{B_2^n} \subset \mathbb{R}^n$ be the set of all $s$-sparse vectors with $\ell_2$-norm at most one.  A matrix $A \in \mathbb{R}^{m' \times n}$ will have the RIP of order $(s,\epsilon)$ for $s \in [n]$ even and $\epsilon \in (0,1)$ if and only if $\sqrt{1-\epsilon} \| {\bf x} - {\bf y} \|_2 \leq \| A {\bf x} - A {\bf y}\|_2 \leq \sqrt{1+\epsilon} \| {\bf x} - {\bf y}\|_2$ holds for all ${\bf x},{\bf y} \in U_{s/2}$.  As in the previous example we will consider $\epsilon$ not too small (e.g., $\epsilon > 0.01$). In this regime it is known that a matrix $A \in \mathbb{R}^{m' \times n}$ with both the RIP of order $(s,\epsilon)$ and as few rows $m'$ as possible will/must have $m' = \Theta \left (s \log (N/s) \right)$.  See, e.g., \cite[Theorem 9.2]{foucart_mathematical_2013} and \cite[Corollary 10.8]{foucart_mathematical_2013} for specific supporting theorems.\\

One natural question concerns whether allowing the matrix $A \in \mathbb{R}^{m' \times n}$ to be replaced by an arbitrary nonlinear function $f:\mathbb{R}^n \rightarrow \mathbb{R}^m$ might permit $m \ll m' = \Theta \left (s \log (N/s) \right)$ to hold while still approximately preserving the $\ell_2$-norms of all $s$-sparse vectors up to the same $\sqrt{1 \pm \epsilon}$-distortion factors.  Toward answering this question one can apply Theorem~\ref{FittingItInTheBall} to see that any function $f:\mathbb{R}^n \rightarrow \mathbb{R}^m$ satisfying $\sqrt{1 - \epsilon} \|{\bf x} - {\bf y}\|_2 \leq \|f({\bf x})-f({\bf y})\|_2 \leq \sqrt{1 + \epsilon} \|{\bf x}- {\bf y}\|_2$
for all ${\bf x}, {\bf y} \in U_{s/2}$ must have 
$m \geq C \frac{(1-\epsilon) \delta^2}{(1+\epsilon)} \frac{\log N(U_{s/2},\delta)}{{\rm diam}^2(U_{s/2})} = C \frac{(1-\epsilon) \delta^2}{4 (1+\epsilon)} \log N(U_{s/2},\delta)$ for our choice of $\delta \in \mathbb{R}^+$.  In order to make a good choice for $\delta$ in this bound we will use the following lemma.

\begin{lemma}[Lemma 10.12 in \cite{foucart_mathematical_2013}]
\label{lem:subsetlemma}
Given integers $s < n$, there exist $N \geq \left( \frac{n}{2s} \right)^{s/4}$ subsets $S_1, \dots, S_N$ of $[n]$ such that $|S_j| = s/2$ and $|S_j \cap S_k| < \frac{s}{4}$ whenever $j \neq k$.
\end{lemma}

We can now use Lemma~\ref{lem:subsetlemma} to define ${\bf x}_1, \dots, {\bf x}_N \in U_{s/2}$ by letting ${\bf x}_j$ be $\sqrt{2/s}$ for each index in $S_j$, and $0$ elsewhere.  Note that $\|{\bf x}_j - {\bf x}_k \|_2 > 1$ will hold whenever $j \neq k$.  As a consequence of the existence of these $N$ points we can see that, e.g., setting $\delta < 1/2$ forces $N(U_{s/2},\delta) \geq N \geq \left( \frac{n}{2s} \right)^{s/4}$.  Hence, choosing any $\epsilon > 0.01$ and, e.g., $\delta = 1/3$ in the bound on $m$ from Theorem~\ref{FittingItInTheBall} above will yield the bound $m \geq C''' s \log (n/2s)$ for a universal constant $C''' \in \mathbb{R}^+$. As a consequence, we can see that linear measurements achieve the RIP with a near optimal level of compression even when compared to arbitrary nonlinear measurement functions, at least for larger values of $\epsilon$.  Allowing $f$ to be nonlinear doesn't help here.
\end{example}

Having digressed with the two previous examples, we hope to have established a clear proof strategy for the manifold case.  If we can lower bound the covering numbers of smooth submanifolds of $\mathbb{R}^n$, then we can use those lower bounds together with Theorem~\ref{FittingItInTheBall} to establish lower bounds of the achievable low-distortion embedding dimension $m$ of any given submanifold $\mathcal{M}$ of $\mathbb{R}^n$ into $\mathbb{R}^m$.  In persuance of that strategy we will develop lower bounds for the covering numbers of submanifolds of $\mathbb{R}^n$ in the next section.

\section{Lower Bounds for the Covering Numbers of Smooth Submanifolds of Euclidean Space}
\label{sec:CoverNumbound}

Our lower covering number bounds for a smooth submanifold $\mathcal{M} \subset \mathbb{R}^n$ will in terms of its dimension $d$, $d$-dimensional volume $V$, and reach $\tau$.  

\begin{definition}[{\bf Reach:} Definition 4.1 in \cite{federer_curvature_1959}]For a subset of Euclidean space $T \subset \mathbb{R}^n$, the reach $\tau$ is defined as 
$$
    \tau(T) = \sup \{ t \geq 0  \big| \, \forall {\bf x} \in \mathbb{R}^n \text{ with } d({\bf x},T) < t, \, {\bf x} \text{ has a unique closest point in } T \}.
$$
\label{def:Reach}
\end{definition}

The following lemma concerns the control of the sectional curvatures of $\mathcal{M}$ by $\tau$, and also demonstrates how $\tau$ can be used to help relate geodesic distances between points on $\mathcal{M}$ to their Euclidean distances.

\begin{lemma}
\label{lem:GeovsEuclid}
Let $\mathcal{M}$ be a compact, smooth submanifold of $\mathbb{R}^n$ possibly with boundary. Furthermore, let $\tau$ be the reach of $\mathcal{M}$, ${\bf p}, {\bf q} \in \mathcal{M}$, ${\bf x} \in \mathbb{R}^n$, and $d$ and $l$ be the Euclidean and geodesic distances between ${\bf p}$ and ${\bf q}$, respectively. Then:
\begin{enumerate}
\item In the interior of $\mathcal{M}$, any sectional curvature $k$ satisfies $\frac{-2}{\tau^2} \leq k \leq \frac{1}{\tau^2}$. 
\item $l - \frac{l^2}{2\tau} \leq d.$ 
\item When restricted to $d \leq \frac{\tau}{2}$, we further have $l \leq d+\frac{2d^2}{\tau}.$
\end{enumerate}
\end{lemma}
\begin{proof}
$1.$ See \cite[proposition A.1]{aamari_estimating_2019}.
$2.$ See \cite[lemma 6.3]{niyogi_finding_2008}. $3.$ See \cite[lemma 6.3]{niyogi_finding_2008} and \cite[lemma 7]{eftekhari_new_2015} which gives $l \leq \tau - \tau \sqrt{ 1 - \frac{2d}{\tau}}$. Using $1 - \sqrt{1-x} \leq \frac{x + x^2}{2}$ we get the claimed $ l \leq d + \frac{2d^2}{\tau}$. 
\end{proof}

Lemma~\ref{lem:GeovsEuclid} in combination with Bishop's volume comparison theorem \cite[Theorem 3.101, part i]{gallot_riemannian_1990} now allows us to prove the main result of this section.

\begin{theorem}
\label{thm:ManifoldCoverLowerBound}
Let $\mathcal{M}$ be a $d$-dimensional smooth submanifold of $\mathbb{R}^n$ possibly with boundary, and with volume $V$ and reach $\tau$. Let $0 < \delta \leq \frac{\tau}{2}$, and $r := \delta (1 + \frac{2\delta}{\tau})$. Let $\omega_{d} = \frac{\pi^{d/2}}{\Gamma(\frac{d}{2}+1)}$ be the volume of the $d$-dimensional Euclidean unit ball $B^{d}_2$. Then the covering number of $\mathcal{M}$ with Euclidean balls centered on $M$ of radius $\delta$, $N(\mathcal{M},\delta)$, satisfies

\begin{align*}
   \displaystyle  \frac{V}
   { \omega_{d}(8\delta)^d} \leq \frac{V}{ \omega_{d}
    (1 + \frac{2\sqrt{2}r}{\tau})^{d-1} r^d}  \leq N(\mathcal{M}, \delta) \hspace{1mm}.
\end{align*}
\end{theorem}

\begin{proof}
We claim that the volume of the intersection of $\mathcal{M}$ with a Euclidean ball having center in $\mathcal{M}$ and radius $\delta$ is no greater than $\omega_{d} (1 + \frac{2\sqrt{2}r}{\tau})^{d-1} r^d$. This claim immediately implies the rightmost inequality.\\

Toward establishing this claim, we note that for two points ${\bf p}, {\bf q} \in \mathcal{M}$ with Euclidean distance $d$, geodesic distance $l$, and having $d \leq \frac{\tau}{2}$, one has $l \leq d(1+\frac{2d}{\tau})$ by Lemma~\ref{lem:GeovsEuclid}. Hence, the intersection of $\mathcal{M}$ with a Euclidean ball of radius $\delta$ centered on $\mathcal{M}$ is contained in a geodesic ball of radius $r = \delta(1 + \frac{2\delta}{\tau})$ with the same center provided that $\delta \leq \frac{\tau}{2}$. Thus, it suffices to bound the volume of such a geodesic ball from above by $w_{d} (1 + \frac{2\sqrt{2}r}{\tau})^{d-1} r^d$ in order to obtain the rightmost inequality.\\

By Lemma~\ref{lem:GeovsEuclid} the sectional curves of $\mathcal{M}$ are bounded below by $\frac{-2}{\tau^2}$. Hence by Bishop's theorem the volume of a geodesic ball of radius $r$ in $\mathcal{M}$ is bounded above by the volume of the geodesic ball of radius $r$ in the hyperbolic space with constant sectional curvature  $\frac{-2}{\tau^2}$. Denote this volume by $V_{\frac{-2}{\tau^2}}(r)$; it is given by the formula 
\begin{align*}
    V_{\frac{-2}{\tau^2}}(r)
    =
     \frac{2\pi^{d/2}}{\Gamma(\frac{d}{2})} 
    \int_{0}^{r} S_{\frac{-2}{\tau^2}}(x)^{d-1} ~dx
\end{align*}
where $S_k(x) = \frac{1}{\sqrt{-k}}\sinh(\sqrt{-k}x)$. Since $\frac{\sinh(x)}{x}$ is increasing on $0 < x$,

\begin{align*}
    V_{\frac{-2}{\tau^2}}(r) 
    &=
     \frac{2\pi^{d/2}}{\Gamma(\frac{d}{2})} 
     \int_{0}^{r} S_{\frac{-2}{\tau^2}}(x)^{d-1} ~dx
    \leq 
    \frac{2\pi^{d/2}}{\Gamma(\frac{d}{2})} 
    \int_{0}^r \left(\frac{\sinh(\frac{\sqrt{2}x}{\tau})}
    {\frac{\sqrt{2} x}{\tau}} \right)^{d-1} x^{d-1} ~dx\\
    &\leq 
    \frac{2\pi^{d/2}}{\Gamma(\frac{d}{2})} 
    \left(\frac{\sinh(\frac{\sqrt{2}r}{\tau})}
    {\frac{\sqrt{2} r}{\tau}} \right)^{d-1} 
    \frac{r^d}{d}.
\end{align*}

Using that $\frac{\sinh(x)}{x} \leq 1 + 2x$ for $0 < x < \sqrt{2} \pi$, we finally obtain that
$$
    V_{\frac{-2}{\tau^2}}(r) \leq \frac{2\pi^{d/2}}{\Gamma(\frac{d}{2})} 
    \left(1 + \frac{2\sqrt{2}r}{\tau} \right)^{d-1} \frac{r^d}{d}
    =
    \omega_{d} \left( 1 + \frac{2\sqrt{2}r}{\tau} \right)^{d-1} r^d. 
$$
This proves the rightmost inequality. To obtain the leftmost inequality we simplify $r = \delta (1 + \frac{2\delta}{\tau})$ via $r \leq 2 \delta \leq \tau$ using that $0 < \delta \leq \frac{\tau}{2}$.  We then get the factor of $8$ from the fact that $(1 + 2\sqrt{2})2 < 8$. 
\end{proof}

In light of Theorem~\ref{thm:ManifoldCoverLowerBound}, the following technical lemma will be useful in choosing the parameter $\delta$ in Theorem~\ref{FittingItInTheBall}.

\begin{lemma} \label{Iknowcalculuskindof}
Let $d,C' \in \mathbb{R}^+$.  If  $\frac{2 C'^{\frac{1}{d}}}{\sqrt{e}} \leq \tau$ then
$$\frac{d}{2e} C'^{\frac{2}{d}} \leq \sup_{\delta \in (0,\tau/2]} \delta^2 \log \left( \frac{C'}{\delta^d} \right).$$
If $\tau \leq \frac{2 C'^{\frac{1}{d}}}{\sqrt{e}}$ then
\begin{align*}
\frac{d \tau^2}{4} \log \left( \frac{2C'^{\frac{1}{d}}}{\tau} \right)\leq \sup_{\delta \in (0,\tau/2]} \delta^2 \log \left( \frac{C'}{\delta^d} \right).
\end{align*}
\end{lemma}
\begin{proof}
One can see that $\delta^2 \log \left( \frac{C'}{\delta^d} \right)$ is maximized at $\delta^* := C'^{\frac{1}{d}}/\sqrt{e} > 0$.  If $\delta^* \notin (0,\tau/2]$ we instead evaluate the function at $\tau / 2$.
\end{proof}

We are now prepared to prove the main theorem of this paper.

\section{The Proof of Theorem~\ref{MAINTHM}}
\label{sec:ProofofMainThm}

Applying Theorem~\ref{FittingItInTheBall} to $\mathcal{M} \subset \mathbb{R}^n$ we learn that
$$m \geq C \left(\frac{1-\epsilon}{1+\epsilon} \right)^2  \, \frac{\delta^2 \log N(\mathcal{M},\delta)}{\mathrm{diam}^2(\mathcal{M})},$$
where we may choose $\delta \in \mathbb{R}^+$ as we wish.  Restricting ourselves to $\delta \in (0,\tau / 2]$ and applying Theorem~\ref{thm:ManifoldCoverLowerBound} we further learn that
$$m \geq C \left(\frac{1-\epsilon}{1+\epsilon} \right)^2  \, \frac{\delta^2 \log \left( \frac{V}
   { \omega_{d}(8\delta)^d} \right) }{\mathrm{diam}^2(\mathcal{M})}.$$
To maximize this lower bound as a function of $\delta$ we now apply Lemma~\ref{Iknowcalculuskindof} with $C' = V/\omega_{d}8^d$.  The stated result now follows.

\bibliographystyle{abbrv}
\bibliography{refs} 

\end{document}